\documentclass[12pt,reqno]{amsart}
\usepackage{amssymb,mathrsfs}
\usepackage[usenames,dvipsnames]{color}
\usepackage{euscript}
\usepackage{multicol}
\usepackage{amssymb}
\usepackage{amsmath}
\usepackage{auto-pst-pdf} 
\usepackage{pstricks} 
\usepackage{pst-plot, pstricks-add} 
\usepackage{times} 

\newcommand{\Q}{\mathcal{Q}}
\newcommand{\eps}{\varepsilon}

\newcommand{\pd}{\partial}
\newcommand{\F}{\mathcal{F}}

\newcommand{\N}{\mathbb{N}}

\newcommand{\R}{\mathbb{R}}

\newtheorem{theorem}{Theorem}[section]
\newtheorem{corollary}[theorem]{Corollary}
\newtheorem{lemma}[theorem]{Lemma}
\newtheorem{proposition}[theorem]{Proposition}

\theoremstyle{definition}

\newtheorem{remark}[theorem]{Remark}

\numberwithin{equation}{section}

\begin{document}

\title[Hardy inequalities for $p-$Laplacians with  Robin boundary conditions]
{Hardy inequalities for ${\bf p-}$Laplacians with  Robin boundary conditions}

\author{Tomas Ekholm}
\address{Tomas Ekholm, Department of Mathematics,
Royal Institute of Technology, S-100 44 Stockholm, Sweden}
\email{tomase@math.kth.se}

\author{Hynek Kova\v r\'{\i}k}
\address{Hynek Kova\v r\'{\i}k, DICATAM, Sezione di Matematica, Universit\`a degli studi di Brescia,
Via Branze, 38 - 25123  Brescia, Italy}
\email{hynek.kovarik@unibs.it}

\author {Ari Laptev}
\address {Ari Laptev, Imperial College London\\ Huxley Building, 180 Queen's Gate \\
London SW7 2AZ, UK 
}
\email{a.laptev@imperial.ac.uk}


\begin{abstract}
In this paper we study the best constant in a Hardy inequality for the $p-$Laplace operator on convex domains 
with Robin boundary conditions. We show, in particular, that the best constant equals $((p-1)/p)^p$ whenever Dirichlet boundary conditions are imposed on a subset of the boundary of non-zero measure. We also discuss some generalizations to non-convex 
domains.
\end{abstract}

\maketitle

{\bf  AMS 2000 Mathematics Subject Classification:} 47F05, 39B72\\

{\bf  Keywords:} $p-$Laplacian, Robin boundary conditions, Hardy inequality \\


\section{\bf Introduction}
\noindent Let $\Omega\subset\R^n$ be an open bounded domain and denote by 
\begin{equation}\label{min}
\delta(x) =  \min_{y\in\partial\Omega} |x-y|
\end{equation}
the distance between a given $x\in\Omega$ and the boundary of $\Omega$. The Hardy inequality for the p-Laplace operator with Dirichlet boundary conditions on $\partial\Omega$:
\begin{equation} \label{hardy-general}
\int_\Omega |\nabla u(x)|^p \,dx \ \geq \ K  \int_{\Omega} \frac{|u(x)|^p}{\delta(x)^p}\, dx, \qquad \forall\ u \in W^{1,p}_0(\Omega), \quad p>1,
\end{equation}
is closely related to the variational problem 
\begin{equation} \label{var-prob-dir}
\mu_p(\Omega)  : =  \inf_{u\in W_0^{1,p}(\Omega)}\, \frac{\int_\Omega |\nabla u(x)|^p \,dx}{\int_{\Omega} \ |u(x)/\delta(x)|^p\, dx}.
\end{equation}
Hence $\mu_p(\Omega)$ is the best possible value of the constant $K$ in \eqref{hardy-general}. Hardy showed in \cite{h} that inequality \eqref{hardy-general} holds with some $K>0$ in dimension one. In higher dimensions it is known, see \cite{ok}, that if $\Omega$ has Lipschitz continuous boundary, then $\mu_p(\Omega)>0$. 
In general, $\mu_p(\Omega)$ depends on the domain $\Omega$ and satisfies the upper bound
$$
\mu_p(\Omega) \leq C_p :=  \Big(\frac{p-1}{p}\Big)^p,
$$
see \cite{MMP}.
However, if $\Omega$ is convex, then $\mu_p(\Omega) = C_p$.
The latter was first proved for $p=n=2$, see \cite[Sec.~5.3]{D3} or \cite[Sec.~1.5]{D1}, then in \cite{MS} for $n=2$ and any $p>1$, and finally in \cite{MMP}
for any $n$ and any $p>1$. Moreover, it was shown in \cite{MMP} that $\mu_2(\Omega) =C_2$ if and only if the variational problem \eqref{var-prob-dir} has no minimiser. The fact that for convex domains there is no minimiser of \eqref{var-prob-dir} opens a possibility to improve inequality \eqref{hardy-general}, even with the sharp constant $K=C_p$, by adding to its right hand side a positive contribution. Such improvements, with various forms of the remainder terms, have been obtained in \cite{A1, A2, AW, BM, FMT, HHL} for $p=2$ and later in \cite{T} for $p\neq 2$. As for non-convex domains, it is known, due to \cite{A}, that in the case $n=p=2$ for simply connected domains one has $\mu_2(\Omega) \geq 1/16$, see also \cite{LS}. For a throughout discussion of various Hardy inequalities for $p=2$ we refer to 
\cite{D2} and references therein. 
\smallskip

In this paper we consider an analogue of the variational problem \eqref{var-prob-dir} for a Robin Laplacian. This means that we replace the numerator of \eqref{var-prob-dir} by the functional 
\begin{equation} \label{form}
\Q_p [\sigma, u] = \int_\Omega |\nabla u|^p\, dx  + \int_{\partial\Omega} \sigma\, |u|^p\, d\nu, \qquad u\in \F(\Omega),
\end{equation}
where $d\nu$ denotes the surface measure on $\partial\Omega$, $\sigma: \partial\Omega\to [0,+\infty]$ is a function which defines the boundary conditions and $\F(\Omega)$ is a suitable family of test functions. The function space $\F(\Omega)$ clearly depends on the choice of $\sigma$. Notice that with the choice $\sigma=+\infty$, and consequently $\F(\Omega)=W_0^{1,p}(\Omega)$,  we arrive at the Dirichlet boundary conditions and hence at problem \eqref{var-prob-dir}. 

To pass from Dirichlet boundary conditions to Robin boundary conditions means to take $\sigma\neq +\infty$. In order to make the choice of $\sigma$ as general as possible we will impose the Dirichlet boundary on a part of the boundary $\Gamma\subseteq\partial\Omega$, which might be empty, and Robin boundary conditions on the remaining part $\partial\Omega\setminus\Gamma$ ;
\begin{equation} \label{lgamma}
\sigma\in \Sigma_\Gamma := \big\{ f :\partial\Omega \to [0,+\infty] \ ,  \ f = +\infty \ \text{on}\ \Gamma,\ f\in L^\infty(\partial\Omega\setminus\overline{\Gamma}) \big\}.
\end{equation}
Consequently, we choose 
$$
\F(\Omega) = W^{1,p}_{0,\Gamma}(\Omega) := \overline{\big\{ u\in C^1(\overline{\Omega})\, :\, u |_{\Gamma} \, =0\big\} } ^{\ \|\cdot\|_{W^{1,p}(\Omega)}}  .
$$
Obviously, the weight function in the denominator of \eqref{var-prob-dir} has to be modified accordingly, since the test functions from $ W^{1,p}_{0,\Gamma}(\Omega)$ do not vanish on the whole $\partial\Omega$. 

In order to define our variational problem we need to introduce some notations.
Let $S$ be the singular set of $\Omega$, i.e.~the set of points in $\Omega$ for which there exist at least two points $y_1, y_2 \in \partial\Omega$ where the minimum in \eqref{min} is achieved. Hence, for $x \in \Omega\setminus S$ let $\pi(x) = y$, where $y$ is the unique point on $\partial \Omega$ satisfying $\delta(x) = |x - y|$. In analogy with the case $p=2$, see \cite{kl}, we then define the function $\alpha:\Omega\setminus S\to [0,+\infty]$ by 
\begin{equation} \label{alphax}
\alpha(x)= \frac{p-1}{p}\, \sigma\big(\pi(x)\big)^{\frac{1}{1-p}}.
\end{equation}
We pass from the weight function $\delta(x)^{-p}$ in \eqref{var-prob-dir} to the weight function
$$
(\delta(x)+ \alpha(x))^{-p},
$$
which takes into account the boundary conditions defined in term of $\sigma$. For example,  if $\sigma=+\infty$, then $\alpha=0$ as expected. Note also that the function is defined almost everywhere in $\Omega$ since the set $S$ has Lebesgue measure zero, see \cite{LN}. Hence we are led to the variational problem 
\begin{equation} \label{var-prob}
\lambda_p(\Omega, \sigma) := \inf_{u\in W^{1,p}_{0,\Gamma}(\Omega)} \frac{\Q_{p}[\sigma, u]}{\| u\|^p_{p,\sigma} }\, ,
\end{equation}
with 
\begin{equation} \label{sigma-norm}
\| u\|_{p,\sigma} = \left(\int_{\Omega} \frac{|u(x)|^p}{(\delta(x)+ \alpha(x))^p}\, dx \right)^{\frac 1p} , \qquad u\in W^{1,p}_{0,\Gamma}(\Omega). 
\end{equation}

\noindent We are going to establish a relation between $\lambda_p(\Omega, \sigma)$ on one hand, and the function $\sigma$ and geometry of $\Omega$ on the other hand.
The main results of this paper are the following:

\section{\bf Main results}

\begin{theorem} \label{thm-1}
Let $\Omega\subset\R^n$ be open bounded and convex with $\partial\Omega$ of class $C^2$. Let $\Gamma\subseteq\partial\Omega$. Then for any 
$\sigma \in \Sigma_\Gamma$ we have $\lambda_p(\Omega, \sigma) \geq C_p$. Moreover, 
\begin{equation} \label{equiv}
\lambda_p(\Omega, \sigma) = C_p \quad \Leftrightarrow \quad \Gamma \neq \emptyset.
\end{equation} 
\end{theorem}

\smallskip

\begin{remark}
Note that Theorem \ref{thm-1} includes also the extreme cases $\Gamma=\emptyset$ and $\Gamma=\pd\Omega$. The first part of the statement, i.e. the inequality $\lambda_p(\Omega, \sigma) \geq C_p$ is proven in Proposition \ref{prop-hardy} which provides a generalization of the Hardy inequality obtained in \cite{kl} for $p=2$. 

The second part of the claim is a consequence of Proposition \ref{prop-sharp}. Equivalence \eqref{equiv} is closely related to the question of the existence of a minimiser for the variational problem \eqref{var-prob}, see Proposition \ref{prop-min}.
\end{remark}

\begin{remark}
Let us comment on the sharpness of the lower bound $\lambda_p(\Omega, \sigma) \geq C_p$. The bound is sharp in the sense that the constant $C_p$ cannot replaced by a bigger one and remain independent of $\sigma$, see section \ref{sec-const} for details. However, if $\Gamma=\emptyset$, then for a given $\sigma\in\Sigma_\Gamma$ Theorem \ref{thm-1} implies that $\lambda_p(\Omega, \sigma) > C_p$. The following Theorem quantifies the gap between $\lambda_p(\Omega, \sigma)$ and $C_p$ in terms of the $\|\sigma\|_{L^\infty(\partial\Omega)}$. 
\end{remark}

\smallskip

\begin{theorem} \label{cor-1}
Let $\Omega$ be as in Theorem \ref{thm-1}. If $\Gamma=\emptyset$, then for any 
$\sigma \in \Sigma_\Gamma$ it holds
\begin{equation} \label{non-sharp}
\lambda_p(\Omega, \sigma) \ \geq \  C_p\, \left(1+ \Big(1+ p\, R_{in}\ \|\sigma\|_{L^\infty(\partial\Omega)}^{\frac{1}{p-1}} \Big)^{-p} \right),
\end{equation}
where 
$$
R_{in} = \sup_{x\in\Omega} \delta(x)
$$
is the in-radius of $\Omega$. 
\end{theorem}

\subsection{Outline of the paper}
We start by the proof of an $L^p$ version of the Hardy inequality for Robin Laplacians, see section \ref{sec-hardy}. Then we provide the proofs of our main results; this is done in section \ref{sec-proofs}. 
In section \ref{sec-conc} we study the behavior of the minimising sequences of the variational problem \eqref{var-prob} in the case when $\lambda_p(\Omega, \sigma) =C_p$, which corresponds to $\Gamma\neq\emptyset$. In particular, we show that minimising sequences, under certain conditions, concentrate on $\Gamma$. Finally, section \ref{sec-outside} is dedicated to the analysis of a hardy-type inequality on a particular non-convex domain, namely on a complement of a ball.

\smallskip

\section{\bf A Hardy inequality}
\label{sec-hardy}

\noindent Similarly as in the case $p=2$, see \cite{kl}, we first establish an appropriate one-dimensional estimate. 

\begin{lemma}\label{lem1}
Let $b>0$ and assume that $u$ belongs to $AC[0,b]$, the space of absolutely continuous functions on $[0,b]$.  
Then for any $\sigma\geq 0$ we have
\begin{align} \label{eq-1d}
\int_0^b |u'(t)|^p \, dt +  \sigma\, |u(0)|^p & \, \geq\,  C_p 
\int_0^b \frac{|u(t)|^p}{(t+\alpha)^p} \, dt  + \frac{(p-1)\, C_p}{(b+\alpha)^p} \int_0^b |u(t)|^p \, dt,
\end{align}
where
\begin{equation} \label{alpha}
\alpha = \frac{p-1}{p}\, \sigma^{\frac{1}{1-p}}.
\end{equation}
\end{lemma}

\begin{proof}  It suffices to prove the inequality for $u>0$. We may assume that $\sigma>0$. 
Let 
$$
f(t) = -(p-1)^{p-1}\, (t+\alpha)^{1-p}
$$ 
and define
\begin{align*}
A & := \Big | \int_0^b f'(t)\, u^p\, dt -  (f(b)-f(0))\, u(0)^p\Big |, \qquad B := \int_0^b |f(b)-f(t)|^{\frac{p}{p-1}}\ u^p(t)\, dt.
\end{align*}
Integration by parts and H\"older inequality show that
\begin{align} \label{upper-1}
A^p & \leq  p^p \Big( \int_0^b |f(b)-f(t)|\, u^{p-1} |u'|\, dt\Big)^p \leq p^p\, B^{p-1}\  \int_0^b |u'|^p\,dt 
\end{align}
On the other hand, the Young inequality gives
\begin{equation} \label{young}
A^p  \geq p\, A\, B^{p-1} -(p-1)\, B^p. 
\end{equation} 
Using the fact that $f$ is negative increasing and that 
$$
(1-s)^{\frac{p}{p-1}} \leq 1-s \leq 1 -s^{\frac{p}{p-1}}  \qquad \forall\ s\in[0,1]
$$ 
we obtain 
$$
B = \int_0^b  |f(t)|^{\frac{p}{p-1}}\left(1 -\frac{|f(b)|}{|f(t)|}\right)^{\frac{p}{p-1}}\ u^p(t)\, dt \leq \int_0^b \Big( |f(t)|^{\frac{p}{p-1}} -|f(b)|^{\frac{p}{p-1}}\Big)\ u^p(t)\, dt.
$$
Moreover, since $u >0$, from the definition of $A$ we get
$$
A \geq  \int_0^b f'(t)\ u^p\, dt + f(0)\, u(0)^p. 
$$ 
The above inequalities in combination with \eqref{young} and \eqref{upper-1} then imply that 
\begin{align*} 
p^p \int_0^b |u'(t)|^p\, dt  - p f(0)\, u(0)^p \geq (p-1)^p\, \int_0^b \frac{u(t)^p}{(t+\alpha)^p} \, dt + (p-1)^{p+1}\, \int_0^b \frac{u(t)^p}{(b+\alpha)^p} \, dt.
\end{align*}
This implies \eqref{eq-1d}. 
\end{proof}

\noindent
If $\Gamma=\emptyset$, then $\sigma\in L^\infty(\partial\Omega)$ and it is easily seen that $W^{1,p}_{0,\Gamma}(\Omega)=W^{1,p}(\Omega)$. 
Mimicking the approach of \cite{kl} we deduce from  Lemma \ref{lem1} the following version of the 
Hardy inequality for Robin Laplacians on $W^{1,p}(\Omega)$.

\begin{proposition} \label{prop-hardy0}
Let $\Omega$ satisfy the hypothesis of Theorem \ref{thm-1}. 
Then for any $\sigma \in L^\infty(\pd\Omega)$ and all $u\in W^{1,p}(\Omega)$ it holds
\begin{align} \label{hardy-aux}
 \Q_p[\sigma, u] 
 & \geq \   C_p  \int_{\Omega} 
\frac{|u(x)|^p}{(\delta(x)+ \alpha(x))^p}\, dx +  (p-1)\, C_p \int_{\Omega}  \frac{|u(x)|^p}{(R_{in}+ \alpha(x))^p}\, dx.
\end{align}
\end{proposition}

\begin{proof}
As in \cite{kl} we first prove inequality \eqref{hardy-aux} for $u\in C^1(\overline{\Omega})$ and $\sigma$ continuous. By Tietze extension theorem then there exists a continuous function $\zeta: \R^n\to \R$ such that 
\begin{equation} \label{zeta}
\zeta \big |_{\pd\Omega} = \sigma. 
\end{equation}
Now let $Q\subset \Omega$ be an open convex polytop  with $N$ sides $\Gamma_j$, $1 \leq j \leq N$. Let $n_j$ be the inner normal vector of the side $\Gamma_j$. 

Let $\delta(x;Q)$ be the distance from $x \in Q$ to the boundary $\partial Q$ and let 
$$
P_j = \big\{ x\in Q\, : \, \exists\, y\in\Gamma_j , \ \delta(x; Q) = |x-y| \big\}.
$$
For each $x \in P_j$ there is a unique $y \in \Gamma_j$ and $t \in [0, t_y]$ for which
\begin{align}\label{x_and_y_notation}
x = y + t \, n_j,
\end{align}
where $t_y$ is chosen in such a way that $y + t_y \, n_j \in \partial P_j$. 
Moreover, we have
\begin{align}
R_{in}(Q) := \sup_{x \in Q} \delta(x;Q) = \max_{1 \leq j \leq N} \sup_{x \in P_j} \delta(x;Q).
\end{align}
Using Lemma \ref{lem1} and \eqref{x_and_y_notation} we get for each $y \in \Gamma_j$ the lower bound
\begin{align} 
\int_0^{t_y} |u'_{n_j}(x)|^p dt + \zeta(y)|u(y)|^p & \geq\  C_p \int_0^{t_y} \frac {|u(x)|^p}{(t + \alpha (x; Q))^p} dt  \label{1D-aux}\\
& \ \ + (p - 1) C_p \int_0^{t_y} \frac {|u(x)|^p}{(t_y + \alpha(x; Q))^p} dt, \nonumber
\end{align}
where 
\begin{equation} \label{alpha-2}
\alpha(x;Q) = \frac{p-1}{p}\, \zeta \big(\pi \big(x; Q\big)\big)^{\frac{1}{1-p}} 
\end{equation} 
and for $x$ in the interior of some $P_j$ we define $\pi(x;Q) = y \in \Gamma_j$, such that $\delta(x;Q) = |x - y|$. Note that $\pi(\cdot;Q)$ is densely defined in $Q$. 

By integrating \eqref{1D-aux} over the boundary $\Gamma_j$ and then summing the resulting inequality over $j=1,..,N$ we arrive at
\begin{align}
\int_{Q} |\nabla u|^p \, dx +  \int_{\pd Q} \zeta(y) |u(y)|^p \, d\nu (y) & \geq\  C_p \int_{Q} \frac {|u(x)|^p}{(\delta(x; Q) + \alpha (x; Q))^p} dx  \label{hardy-q} \\
& \ \  \ + (p - 1) \, C_p \int_{Q} \frac {|u(x)|^p}{(R_{in}(Q) + \alpha(x;Q))^p} dx . \nonumber 
\end{align}
From the convexity of $\Omega$ it follows that there exits a sequence of convex polytops  $Q_m\subset\Omega,\, m\in\N$, which approximates $\Omega$. More precisely, for every $\varepsilon$ there exists an $m_\eps$ such that the Hausdorf distance between $\Omega$ and $Q_{m_\eps}$ satisfies $d_H(\Omega, Q_{m_\eps}) < \varepsilon$. Similarly as in \cite{kl} we then conclude, using the continuity of $\zeta$ in combination with \eqref{zeta},  that 
$$
\zeta \big(\pi \big(x; Q_m \big)\big) \, \to  \, \sigma(\pi(x)) \quad m\to\infty,  \qquad \text{a. e.} \quad  x\in\Omega.
$$
Hence by the continuity of $u$ 
$$  
\int_{\partial Q_m}\! \zeta (y) \, |u(y)|^p\, d y\ \to \  \int_{\partial\Omega}\! \sigma(y) \, |u(y)|^p\, d\nu(y)
$$
as $m\to\infty$. The last two equations together with \eqref{alpha-2}, dominated convergence theorem and the fact that $R_{in}(Q_m) \leq R_{in}$ for every $m$ imply that 
\begin{align} \label{hardy-cont}
\Q[\sigma,u] & \geq  C_p \int_{\Omega} \frac {|u(x)|^p}{(\delta(x) + \alpha (x))^p} dx  + (p - 1) \, C_p \int_{\Omega} \frac {|u(x)|^p}{(R_{in} + \alpha(x))^p} dx , \quad u\in C^1(\overline{\Omega})
\end{align}
holds for all $\sigma$ continuous. 

\noindent Now if $\sigma\in L^\infty(\partial\Omega)$, then in view of the regularity of $\partial\Omega$ there exists a sequence of continuous functions $\sigma_k$ on $\partial\Omega$ which converges to $\sigma$ in $L^\infty(\partial\Omega)$  as $k\to\infty$. From inequality \eqref{hardy-cont} it follows that \eqref{hardy-aux} holds for  all $\sigma_k$.  
Since $u |_{\partial\Omega} \in L^p(\partial\Omega,d\nu)$ for any $u\in C^1(\overline{\Omega})$, using  the dominated convergence  we obtain \eqref{hardy-aux} for any $\sigma\in L^\infty(\partial\Omega)$ and all $u\in C^1(\overline{\Omega})$. 

Finally, let $u\in W^{1,p}(\Omega)$. By density there exists a sequence $u_j \in C^1(\overline{\Omega})$ such that $u_j \to u$ in  $W^{1,p}(\Omega)$ as $j\to \infty$. In view of the regularity of $\Omega$ it follows that $W^{1,p}(\Omega) \hookrightarrow L^p(\partial\Omega)$ with compact imbedding,  see \cite[Sect.7.5]{ad}. Hence, after applying inequality \eqref{hardy-cont} to $u_j$ and letting $j\to \infty$ we conclude that \eqref{hardy-aux} holds for all $u\in W^{1,p}(\Omega)$.  
\end{proof}

\begin{remark}
In the situation when $\sigma$ is constant, a simpler proof of \eqref{hardy-aux}, without the second term on the right hand side, can be given, see \cite[Lem. 4.4]{k} for the case $p=2$ and \cite[Lem. 3.1]{dpg} for the case $p>1$. 
\end{remark}

\noindent As an immediate consequence of the above Proposition we obtain

\begin{proposition}\label{prop-hardy} 
Let $\Omega$ satisfy the hypothesis of Theorem \ref{thm-1}. 
Then for any $\sigma \in \Sigma_\Gamma$ and all $u\in W_{0,\Gamma}^{1,p}(\Omega)$ it holds
\begin{align} \label{hardy-eq}
 \Q_p[\sigma, u] 
 & \geq \   C_p  \int_{\Omega} 
\frac{|u(x)|^p}{(\delta(x)+ \alpha(x))^p}\, dx +  (p-1)\, C_p \int_{\Omega}  \frac{|u(x)|^p}{(R_{in}+ \alpha(x))^p}\, dx.
\end{align}
\end{proposition}

\begin{proof}
Let $u\in W_{0,\Gamma}^{1,p}(\Omega)$ and define the sequence $\{\sigma_n\}_{n\in\N} \subset L^\infty(\pd\Omega)$ by 
$$
\sigma_n(y) = \left\{
\begin{array}{l@{\qquad}l}
 \tilde\sigma(y)    &       {\rm if \ } \  \  y\in\  \partial\Omega\setminus\Gamma ,  \\
n     &    {\rm if \ } \   \ y\in\Gamma.
\end{array}
\right. 
$$
Proposition \ref{prop-hardy0} now implies 
\begin{equation} \label{hn}
\Q[\sigma_n,u]  \geq  C_p \int_{\Omega} \frac {|u(x)|^p}{(\delta(x) + \alpha_n(x))^p}\, dx  + (p - 1) \, C_p \int_{\Omega} \frac {|u(x)|^p}{(R_{in} + \alpha_n(x))^p} \, dx , 
\end{equation}
where
$$
\alpha_n(x) = \frac{p-1}{p}\, \sigma_n \big(\pi(x)\big)^{\frac{1}{1-p}} , \qquad x\in\Omega.
$$
Since $\Q[\tilde \sigma,u] = \Q[\sigma_n,u]$ and   
$$
 \alpha_n(x) \leq  \alpha_{n+1}(x) \qquad \forall\ n\in\N, \quad x \in\Omega,
$$
the statement follows from \eqref{hn} by monotone convergence. 
\end{proof}

\noindent The following corollary of Proposition \ref{prop-hardy} provides yet another improvement of the Hardy inequality \eqref{hardy-general} with the sharp constant $K=C_p$. 

\begin{corollary} \label{cor-0}
For any $u\in W^{1,p}_0(\Omega)$ it holds
\begin{equation} \label{hardy-imp}
\int_\Omega |\nabla u(x)|^p \,dx \ \geq \   C_p  \int_{\Omega} \frac{|u(x)|^p}{\delta(x)^p}\, dx + \frac{(p-1)\,  C_p}{R_{in}^p} \  \|u\|_{L^p(\Omega)}^p. 
\end{equation}
\end{corollary}

\begin{proof}
It suffices to apply Proposition \ref{prop-hardy} with $\Gamma = \pd\Omega$. 
\end{proof}

\medskip

\section{\bf Proofs of the main results}
\label{sec-proofs}

\noindent We start with the following Proposition which provides sufficient conditions for the existence of a minimizer of the variational problem \eqref{var-prob}.

\begin{proposition} \label{prop-min}
Let $\Omega\subset\R^n$ be open and bounded with $\partial\Omega$ of class $C^2$.
Assume that $\sigma \in L^\infty (\partial \Omega)$. Then \eqref{var-prob} admits a minimiser. In other words, there exists $\psi\in W^{1,p}(\Omega), \ \psi \neq 0,$ such that
\begin{equation} \label{psi-min}
\lambda_p(\Omega, \sigma) = \frac{\Q_{p}[\sigma, \psi]}{\|\psi\|_{p,\sigma}^p }\ . 
\end{equation}
\end{proposition}

\begin{proof}
Let $\{u_j\}_{j\in\N}$ be a minimising sequence for $\lambda(\Omega,\sigma)$. Assume that 
\begin{equation} \label{normalised}
\|u_j\|_{p,\sigma}^p = \int_{\Omega} (\delta(x)+ \alpha(x))^{-p} \ |u_j(x)|^p\, dx=1 \quad \forall\ j\in\N. 
\end{equation}
Since $\{u_j\}$ is bounded in $W^{1,p}(\Omega)$, there exists a subsequence, which we still denote by $u_j$ and a function $\psi\in W^{1,p}(\Omega)$ such that $u_j \to \psi$ weakly in $W^{1,p}(\Omega)$. 
In view of the regularity of $\Omega$ and the compactness of the imbedding  $W^{1,p}(\Omega) \hookrightarrow L^p(\Omega)$ we may suppose (by passing to a subsequence if necessary) that $u_j$ converges strongly to $\psi$ in $L^p(\Omega)$. Moreover, since $W^{1,p}(\Omega)$ is compactly imbedded also in $L^p(\partial\Omega)$, see e.g. \cite[Thm.5.22]{ad}, it follows that we can find a subsequence $\{v_j\} \subset \{u_j\}$ such that 
$v_j |_{\partial\Omega} \to \psi|_{\partial\Omega}$ almost everywhere on $\partial\Omega$. By the weak lower semicontinuity of $\int_\Omega |\nabla u|^p$ and  the Fatou Lemma we thus obtain 
\begin{equation} \label{semi-cont}
\liminf_{j\to\infty} \Q_p[\sigma, v_j] \geq \Q_p[\sigma, \psi].
\end{equation}
On the other hand, 
$$
\| (\delta+ \alpha)^{-p}\|_{L^\infty(\Omega)} = \Big(\frac{p}{p-1}\Big)^p\ \|\sigma\|_{L^\infty(\partial\Omega)}^{\frac{p}{p-1}} \ < \ \infty,
$$
see \eqref{alphax}.
The strong convergence of $v_j  \to \psi$ in $L^p(\Omega)$ thus implies that
$$
\|\psi\|_{p,\sigma}^p= \int_{\Omega} (\delta(x)+ \alpha(x))^{-p} \ |\psi(x)|^p\, dx=1 .
$$
Hence $\psi\neq 0$ and in view of \eqref{semi-cont} we have
$$
\Q_{p}[\sigma, \psi] \geq \lambda(\Omega, \sigma) = \liminf_{j\to \infty} \Q_p[\sigma, v_j] \geq \Q_{p}[\sigma, \psi].
$$
This implies \eqref{psi-min}.
\end{proof}

\begin{proof}[\bf Proof of Theorem \ref{cor-1}]
Let $\psi$ be a minimiser for $\lambda(\Omega, \sigma)$ whose existence is guaranteed by Proposition \ref{prop-min}. 
By \eqref{hardy-eq}  we have 
\begin{align*}
\lambda_p(\Omega, \sigma) & =\frac{\Q_{p}[\sigma, \psi]}{\|\psi\|_{p,\sigma}^p}
  \geq C_p\, \left(1 +(p-1)\,  \min_{x\in\Omega} \left(\frac{\delta(x)+\alpha(x)}{R_{in} + \alpha(x)} \right)^p\, \right). 
\end{align*}
Since 
$$
 \frac{\delta(x)+\alpha(x)}{R_{in} + \alpha(x)} \ \geq\ \frac{\min_{x\in\Omega} \alpha(x)}{R_{in}+\min_{x\in\Omega} \alpha(x)} \qquad \forall\ x\in\Omega,
$$
the lower bound \eqref{non-sharp} follows from \eqref{alphax}. 
\end{proof}

\noindent In order to give a proof of Theorem \ref{thm-1} we need the following

\begin{proposition} \label{prop-sharp}
Let $\Omega\subset\R^n$ be open and bounded with $\partial\Omega$ of class $C^2$.
If $\Gamma\neq\emptyset$, then $\lambda_p(\Omega, \sigma) \leq C_p$.
\end{proposition}

\begin{proof}
By assumption there exists $y_0\in \Gamma$ and an $r>0$ such that $\alpha(x) =0$ on $B(y_0,r)\cap \Omega$. 
Let $\eps>0$ and introduce the following continuous functions
$$
f_{\eps} (x) = \left\{
\begin{array}{l@{\qquad}l}
\eps                   & {\rm if \ } |x-y_0| \leq r, \\
\text{linear in} \ |x| & {\rm if \ } r \leq |x-y_0| \leq r + \eps, \\
\frac 1p               & {\rm if \ } r + \eps \leq |x-y_0|,
\end{array}
\right. 
$$
and
\begin{equation} \label{u-eps}
u_{\eps} (x) = \left\{
\begin{array}{l@{\qquad}l}
\delta (x)^{f_\eps (x) + 1 - \frac 1p} & {\rm if \ } 0 \leq \delta(x)\leq \eps, \\
\text{linear in} \ \delta(x)           & {\rm if \ } \eps \leq \delta(x) \leq 2 \eps, \\
0                                      & {\rm if \ } 2\eps \leq \delta(x).
\end{array}
\right. 
\end{equation}

\begin{figure}[h] \label{fig1}
\begin{center}
\begin{pspicture}(-5,-1)(5,5)
	\psarc[fillstyle=solid, fillcolor=gray](0,0){3}{0}{180}
	\psarc[fillstyle=solid, fillcolor=white, linecolor=white](0,1){3}{0}{180}
	\psline(-5,0)(5,0)
	\put(5.3, -0.1){$\partial \Omega$}
	\psarc(0,0){3}{0}{180}
	\psarc(0,0){4}{0}{180}
	\psline(-5,1)(5,1)
	\put(5.3, 0.9){$\varepsilon$}
	\psline(-5,2)(5,2)
	\put(5.3, 1.9){$2 \varepsilon$}
	\pscircle*(0, 0){2pt}
	\put(-0.15,-0.5){$y_0$}
	\psline{->}(0,0)(1,2.85)
	\psline{->}(0,0)(-1,3.85)
	\put(1.1,3){$r$}
	\put(-1.3,4.1){$r + \varepsilon$}
\end{pspicture}
\end{center}
\caption{}
\end{figure}

\noindent  To proceed we introduce the following notation: 
$$
\Omega_\eps : = \{ x\in\Omega\, : \, \delta(x) \leq \eps\}, \quad E(\eps, r) := B(y_0,r)\cap \Omega_\eps, \quad D(y_0, r) :=  B(y_0,r)\cap \partial\Omega.
$$
Notice that $E(\eps,r)$ is the set in Figure 1 marked in grey. 
By \cite[Sec. I.3]{se} there exists a set of coordinates $(\delta, \omega)\in\R^n$ such that the transformation $x \to (\delta(x), \omega(x))$ is $C^1$ on $\Omega_\eps$ for $\eps$ sufficiently small. Moreover,  the Jacobian $J(\delta,\omega)$ of this transformation satisfies 
\begin{equation} \label{jacob}
\lim_{\delta\to 0} J(\delta,\omega) = 1.
\end{equation}
From \eqref{u-eps} and \eqref{jacob} we obtain
\begin{align}
\eps \int_{E(\eps,r)} |\nabla u_\eps(x)|^p\, dx &= \left(\eps+1 -\frac 1p\right)^p\, \eps  \int_{E(\eps,r)} \delta(x)^{\, p\eps -1}\, dx \nonumber \\
& = \left(\eps+1 -\frac 1p\right)^p\, \eps \int_{D(y_0,r)}  \int_0^\eps \delta^{\, p\eps-1}\, J(\delta,\omega)\, d\delta\, d\omega \label{grad} \\
& =  \frac{C_p}{p}\ \nu(D(y_0, r))\, (1 + o_\eps(1)) , \nonumber
\end{align}
where $o_\eps(1)$ denotes a quantity which tends to zero as $\eps\to 0$. Similarly we find for $\eps\to 0$
\begin{align}
\eps \int_{E(\eps,r)} \frac{ |u_{\eps}(x)|^p}{(\delta(x)+ \alpha(x))^{p}} \, dx   &= \eps \int_{E(\eps,r)}\frac{ |u_{\eps}(x)|^p}{\delta(x)^{p}} \, dx = \eps \int_{E(\eps,r)} \delta(x)^{\, p\eps -1}\, dx \nonumber \\
 & =  \frac{1}{p}\ \nu(D(y_0, r))\, (1 + o_\eps(1)), \label{weighted}
\end{align}
On the other hand, using the fact that $|\nabla f_\eps| \leq C/\eps$ for some $C>0$ in combination with \eqref{jacob} it is straightforward to verify that 
$$
\lim_{\eps\to 0}\,  \eps \int_{\Omega\setminus E(\eps,r)} |\nabla u_\eps(x)|^p\, dx\,  =\,  \lim_{\eps\to 0} \, \eps \int_{\Omega\setminus E(\eps,r)}  \frac{ |u_{\eps}(x)|^p}{(\delta(x)+ \alpha(x))^{p}} \, dx = 0.
$$
Hence by collecting the above results we arrive at 
$$
\lambda_p(\Omega, \sigma) \leq \lim_{\eps \to 0} \frac{\Q_{p}[\sigma, u_\eps]}{\|u_\eps\|_{p,\sigma}^p } = C_p,
$$
and the claim follows.
\end{proof}

\begin{proof}[\bf Proof of Theorem \ref{thm-1}]
The inequality $\lambda_p(\Omega,\sigma) \geq C_p$ follows from Proposition \ref{prop-hardy}. The equivalence \eqref{equiv} follows from Theorem \ref{cor-1} and Proposition \ref{prop-sharp}.
\end{proof}

\subsection{The case of constant $\sigma$}
\label{sec-const}

\noindent Here we provide a more detailed information about the quantity $\lambda_p(\Omega,\sigma)$ in the case when $\sigma$ is a positive constant. 

\begin{proposition} 
Let $\Omega\subset\R^n$ be convex and bounded. Then
\begin{align}
\lim_{\sigma\to 0+} \lambda_p(\Omega, \sigma)  & =+ \infty, \label{lim-N} \\
\lim_{\sigma\to \infty} \lambda_p(\Omega, \sigma)  & = C_p.  \label{lim-D} \\
\inf_{\Omega \ \text{convex}} \lambda_p(\Omega, \sigma) & = C_p. \label{inf}
\end{align}
\end{proposition}

\begin{proof}
To prove \eqref{lim-N} we first note that that there exists a constant $c$, depending only on $R_{in}$, such that for all $\sigma\leq 1$ and all $x\in\Omega$ we have $(\delta(x) +\alpha)^p \geq c\, \sigma^{\frac{p}{1-p}}$. Hence 
\begin{align*} 
\lambda_p(\Omega, \sigma) & \geq c \,\sigma^{\frac{p}{1-p}}  \inf_{u\in W^{1,p}(\Omega)}\, \frac{\Q_{p}[\sigma, u]}{\|u\|^p_{L^p(\Omega)}} \, \geq \,c\, \sigma^{\frac{1}{1-p}}  \inf_{u\in W^{1,p}(\Omega)}\, \frac{\Q_{p}[1, u]}{\|u\|^p_{L^p(\Omega)}} \ \geq \ \tilde c\ \sigma^{\frac{1}{1-p}}. 
\end{align*}
holds for all $\sigma\leq 1$.
This proves \eqref{lim-N}. To prove \eqref{lim-D} let $u_j \in W_0^{1,p}(\Omega)$ be a minimising sequence for the variational problem \eqref{var-prob-dir}. Since $\alpha\to 0$ as $\sigma\to\infty$, the monotone convergence shows that
\begin{align*}
\limsup_{\sigma\to\infty}  \lambda_p(\Omega, \sigma) & \leq \limsup_{\sigma\to\infty} \frac{\Q_{p}[\sigma, u_j]}{\int_{\Omega} 
(\delta(x)+ \alpha)^{-p} \ |u_j(x)|^p\, dx}  = \frac{\int_\Omega |\nabla u_j(x)|^p \,dx}{\int_{\Omega} \ |u_j(x)/\delta(x)|^p\, dx}
\end{align*}
holds for all $j$. By letting $j\to \infty$ we get 
$$
\limsup_{\sigma\to\infty}  \lambda_p(\Omega, \sigma)  \leq C_p.
$$
This in combination with \eqref{non-sharp} implies \eqref{lim-D}. 

\smallskip

\noindent Finally, to prove \eqref{inf} we consider the example $\Omega=B_R$, i.e. the the ball centered at origin with radius $R$. Let  
\begin{align*}
u_R(x) = (R + \alpha - |x|)^{(p-1)/p}.
\end{align*}
Then
\begin{align*}
\lambda_p(B_R, \sigma) 
\leq C_p + \frac{ \sigma \alpha^{p - 1} R^{n - 1}}{\int_0^R r^{n-1} (R + \alpha - r)^{-1} dr}.
\end{align*}
Since 
$$
\lim_{R\to \infty}  \frac{ \sigma \alpha^{p - 1} R^{n - 1}}{\int_0^R r^{n-1} (R + \alpha - r)^{-1} dr} = 0,
$$
this shows that 
$$
\inf_{\Omega \ \text{convex}} \lambda_p(\Omega, \sigma)  \leq C_p.
$$
The opposite inequality follows from Theorem \ref{cor-1}.
\end{proof}

\section{\bf Concentration effect}
\label{sec-conc}

\noindent In this section we are going to study the properties of the minimizing sequences of the problem \eqref{var-prob} in the case $\Gamma\neq \emptyset$. Consider first the (normalized) minimizing sequence constructed in the proof of Proposition \ref{prop-sharp}. More precisely, let 
$$
v_n = n^{-\frac 1p}\, u_{1/n}, \qquad n\in\N,
$$ 
where $u_\eps$ is given by \eqref{u-eps}. In view of \eqref{grad} and \eqref{weighted} it is straightforward to verify that
\begin{equation}  \label{seq-cond}
v_n \xrightarrow{w} 0 \quad \text{in } \ W_{0,\Gamma}^{1,p}(\Omega), \quad \text{and} \quad  \liminf_{n\to\infty} \|v_n\|_{p,\sigma} >0.
\end{equation}
Moreover, we observe that $v_n$ concentrates at $\Gamma$. Indeed, we have 
\begin{equation} \label{concentration}
\nabla v_n\ \to 0 \quad \text{in} \quad  L^p_{loc}(\Omega).
\end{equation}
Below we are going to show that {\it any} minimizing sequence satisfying \eqref{seq-cond} concentrates at $\Gamma$ in the sense of \eqref{concentration}.

\begin{theorem} \label{thm-conc}
Let $v_n$ be a minimizing sequence for the problem \eqref{var-prob}. Assume that $v_n$ satisfies \eqref{seq-cond}. Then 
\begin{equation} \label{eq-conc}
\int_M |\nabla v_n|^p \, \to \, 0
\end{equation}
for any compact set $M \subset \overline{\Omega}\setminus\overline{\Gamma}$. 
\end{theorem}

\begin{proof}
Let $n_y$ denote the inner normal vector to $\partial\Omega$ at a point $y\in\partial\Omega$. For a given $\eps>0$ we define 
\begin{equation} \label{omega-eps}
\Omega_\eps = \{ x\in\Omega\, :\, \exists \, t\in [0,\eps], \ \exists\, y\in\overline{\Gamma} \ :\ x=y +t \,n_y\}.
\end{equation}
From the regularity assumptions on $\partial\Omega$ it follows that $\Omega_\eps$ is not self-intersecting for $\eps$ small enough. 

Suppose now that \eqref{eq-conc} is false. Then there exists a compact set $K \subset \overline{\Omega}\setminus\overline{\Gamma}$ and a number $\gamma$ such that 
\begin{equation} \label{false}
\liminf_{n\to\infty} \int_K |\nabla v_n|^p \, \geq \, \gamma.
\end{equation} 
Let us now take $\eps$ small enough such that $K \subset \Omega_\eps' : =\overline{\Omega}\setminus\Omega_\eps$. This is possible due to the assumption on $K$. From the boundedness of $v_n$ in $W^{1,p}(\Omega)$ and from the Hardy inequality we infer that 
\begin{equation} \label{w-infty}
\sup_n \|v_n\|_{p,\sigma} < \infty. 
\end{equation}
By the Rellich-Kondrashov theorem and the first part of \eqref{seq-cond} 
$$
v_n \to 0 \quad \text{in} \ \  L^p_{loc}(\Omega).
$$
Moreover, $(\delta+\alpha)^{-1} \in L^\infty(\Omega_\eps')$. Hence in view of \eqref{w-infty} we have
\begin{equation} \label{aux}
a_n := \int_{\Omega_\eps'} \Big |\frac{v_n}{\delta+\alpha}\Big |^p \to 0  
\end{equation}
We thus obtain the following lower bound: 
$$
\frac{\Q_p[\sigma,v_n]}{\|v_n\|^p_{p,\sigma}}\,  \geq\,  \frac{\int_{\Omega_\eps} |\nabla v_n|^p +\gamma}{\int_{\Omega_\eps} \left |\frac{v_n}{\delta+\alpha}\right |^p +a_n}.
$$
Following \cite{MMP} we now pass to the coordinates $(\delta,\omega)$ in $\Omega_\eps$. Using the one-dimensional Hardy inequality and \eqref{jacob} we find that 
\begin{align*}
\int_{\Omega_\eps} |\nabla v_n|^p & \geq  \int_\Gamma \int_0^\eps |\partial_\delta\, v_n|^p\, J(\delta,\omega)\, d\delta\, d\omega \\
& \geq (1+o(1))\, C_p  \int_\Gamma \int_0^\eps |v_n/\delta|^p\, J(\delta,\omega)\, d\delta\, d\omega \\
& = (1+o(1))\, C_p  \int_{\Omega_\eps} |v_n/\delta|^p ,
\end{align*}
where $o(1)$ denotes a quantity which tends to zero as $\eps\to 0$.
Hence for $\eps$ small enough we have
$$
\liminf_{n\to\infty}  \frac{\Q_p[\sigma,v_n]}{\|v_n\|^p_{p,\sigma}}\, \geq \, (1+o(1))\, C_p +\frac{\gamma}{\sup_n \|v_n\|^p_{p,\sigma} } \,  > \, C_p\, , 
$$
see \eqref{w-infty} and \eqref{aux}. This is in contradiction with the fact that $v_n$ is a minimizing sequence.  
\end{proof}

\begin{remark}
The concentration effect in the case $\Gamma=\pd\Omega$ was proved in \cite{MMP}.
\end{remark}


\section{\bf Hardy inequality on a complement of a ball} 
\label{sec-outside}
In this section we are going to discuss the validity of a Hardy-type inequality for the functional \eqref{form} on a particular non-convex domain, namely on a complement of a ball in $\R^n$. Let us denote by $B_R^c$ the complement in $\R^n$ of the ball of radius $R$ centered in the origin. 

\smallskip
\noindent
The following result is certainly not new, but we prefer to give its proof for the sake of completeness. 

\begin{proposition} \label{prop-hardy-rn}
Assume that $n>p$. Then the  inequality 
\begin{equation}\label{hardy-rn-class}
\int_{B_R^c} |\nabla u |^p \ \geq \ \left(\frac{n-p}{p}\right)^p\, \int_{B_R^c} \frac{|u|^p}{|x|^p} \
\end{equation}
holds true for all $u\in W^{1,p}(B_R^c)$ and any $R>0$. 
\end{proposition}

\begin{proof}
By density and by the inequality $|\nabla u (x)| \geq |\nabla |u(x) ||$, which holds for almost every $x\in B_R^c$, see e.g.~\cite[Thm.~6.17]{LL}, it suffices to prove the inequality for all positive functions $u\in C^\infty(B_R^c)$ supported in a compact set containing $B_R$. Moreover, in view of the rearrangement inequalities, see \cite[Thm.~3.4]{LL}, we may assume without loss of generality  that $u$ is radial, i.e.~$u(x) = f(|x|)$, where $f \in C^\infty([R,\infty)$ is non-negative and such that for some $\rho >R$ we have 
$$
r \geq  \rho \quad \Rightarrow \quad f(r)=0.
$$
Integration by parts together with the H\"older inequality then imply 
\begin{align*}
\int_R^\infty \frac{f(r)^p}{r^p}\, r^{n-1}\, dr & = \frac{1}{n-p} \, \big [ f(r)^p\, r^{n-p}\big]_R^\rho - \frac{p}{n-p} \int_R^\infty f(r)^{p-1}\, f'(r)\, r^{n-p}\, dr \\
 & \leq  \frac{p}{n-p}\int_R^\infty f(r)^{p-1}\, r^{\frac{(n-p-1)(p-1)}{p}}\ |f'(r)|\ r^{\frac{n-1}{p}}\, dr \\
 & \leq \frac{p}{n-p} \left ( \int_R^\infty \frac{f(r)^p}{r^p}\, r^{n-1}\, dr\right)^{\frac{p-1}{p}}\ \left(\int_R^\infty |f'(r)|^p\ r^{n-1}\, dr\right)^{\frac 1p}\, ,
\end{align*}
where we have used the positivity of $f$ in the second line. Hence
$$
\int_R^\infty \frac{f(r)^p}{r^p}\, r^{n-1}\, dr \ \leq \ \left(\frac{p}{n-p}\right)^p\, \int_R^\infty |f'(r)|^p\ r^{n-1}\, dr,
$$
and the claim follows.
\end{proof}

\noindent It is not difficult to verify that the constant $\left(\frac{n-p}{p}\right)^p$ cannot be improved and that inequality \eqref{hardy-rn-class} fails if $p\geq n$. It turns out that when we replace the left hand side by the functional \eqref{form} with $\sigma$ constant and positive, then \eqref{hardy-rn-class}, with a different constant,  extends also to the case $p > n$. 

\begin{theorem} \label{thm-outside}
Assume that $p>n$ and that $\sigma>0$. Then the inequality 
\begin{equation} \label{hardy-outside}
\int_{B_R^c} |\nabla u |^p  + \sigma\! \int_{\partial B_R} |u|^p  \ \geq \ C(\sigma, R) \, \int_{B_R^c} \frac{|u|^p}{|x|^p}, 
\end{equation}
with
\begin{equation}\label{constant}
C(\sigma, R) = \, \min\left\{   \left(\frac{p-n}{p}\right)^p,\  R^{\, p}\, \sigma^{\frac{p}{p-1}}\right\}\, .
\end{equation}
holds for all $u\in W^{1,p}(B_R^c)$.
\end{theorem}

\begin{proof}
Let $\delta(x) = |x|-R$ and let 
$$
\gamma := \frac{p-n}{p}\, \sigma^{\frac{1}{p-1}}\, .
$$
From the convexity of the function $|x|^p$ in $\R^n$ it follows that 
\begin{equation}\label{conv-aux}
|\xi_1|^p \ \geq \ |\xi_2|^p + p\, |\xi_2|^{p-2}\ \xi_2 \cdot(\xi_1-\xi_2)
\end{equation}
holds for all $\xi_1, \xi_2\in\R^n$.  We apply \eqref{conv-aux} with 
$$
\xi_1 = \nabla u, \qquad \xi_2 = \frac{\beta\, u\, \nabla \delta}{\delta+\gamma},
$$
where $\beta>0$ is a parameter whose value will be specified later. Hence 
\begin{align} \label{eq-basic}
\int_{B_R^c} |\nabla u |^p & \geq \int_{B_R^c} \frac{\beta^p\, |u |^p}{(\delta+\gamma)^p} + p \int_{B_R^c} \frac{ \beta^{p-1} |u |^{p-1}}{(\delta+\gamma)^{p-1}}\, \nabla\delta \cdot\left(\nabla u - \frac{\beta\, u\, \nabla \delta}{\delta+\gamma}\right) 
\end{align}
Since $ |\nabla \delta| =1, \ \Delta \delta = \frac{n-1}{|x|}$ and since the normal derivative of $\delta$ is equal to $-1$ on $\partial B_R$, an integration by parts gives
\begin{align*} 
 \int_{B_R^c} \frac{|u |^{p-1}}{(\delta+\gamma)^{p-1}}\, \nabla\delta \cdot \nabla u & = -\gamma^{1-p} \int_{\partial B_R} |u|^p - (n-1)\int_{B_R^c} \frac{ |u |^p}{|x|\ (\delta+\gamma)^{p-1}} \\
&\quad +(1-p) \int_{B_R^c} \frac{|u |^{p-1}}{(\delta+\gamma)^{p-1}}\, \nabla\delta \cdot \nabla u  + (p-1) \int_{B_R^c} \frac{ |u |^p}{(\delta+\gamma)^p} .
\end{align*}
This in combination with \eqref{eq-basic} yields
\begin{align} 
\int_{B_R^c} |\nabla u |^p + \gamma^{1-p} \beta^{p-1} \int_{\partial B_R} |u|^p & \geq 
(p-1)(\beta^{p-1} -\beta^p) \int_{B_R^c}\frac{ |u |^p}{(\delta+\gamma)^p} \nonumber \\
& - (n-1) \, \beta^{p-1} \int_{B_R^c} \frac{ |u |^p}{|x|\ (\delta+\gamma)^{p-1}} \label{eq-basic1}
\end{align}
Assume now that $R> \gamma$ in which case $\left(\frac{p}{p-n}\right)^p\, R^{\, p}\, \sigma^{\frac{p}{p-1}} <1$. 
Then 
$$
\delta(x) +\gamma = |x|-R+\gamma  <  |x|
$$ 
and  the above inequality gives
\begin{align} 
\int_{B_R^c} |\nabla u |^p + \gamma^{1-p} \beta^{p-1} \int_{\partial B_R} |u|^p &\ \geq \
\Big( (p-n)\beta^{p-1} -(p-1)\beta^p\Big) \int_{B_R^c}\frac{ |u |^p}{|x|^p} \, .\label{eq-basic2}
\end{align}
The constant in front of the integral on the right hand side attains its maximum for 
\begin{equation} \label{beta}
\beta= \frac{p-n}{p}.
\end{equation}
Inserting this value of $\beta$ into \eqref{eq-basic2} we obtain \eqref{hardy-outside} in the case $R< \gamma$. 

\smallskip

\noindent If $R \geq \gamma$, then we have $\left(\frac{p}{p-n}\right)^p\, R^{\, p}\, \sigma^{\frac{p}{p-1}} \geq 1$ and 
$$
\delta(x) +\gamma = |x| -R + \gamma \ \leq \ \frac{\gamma}{R}\ |x|.
$$
Inequality \eqref{eq-basic1} then implies that 
\begin{align*} 
\int_{B_R^c} |\nabla u |^p + \gamma^{1-p} \beta^{p-1} \int_{\partial B_R} |u|^p &\ \geq \
\Big( (p-n)\beta^{p-1} -(p-1)\beta^p\Big) \left(\frac R\gamma\right)^p \int_{B_R^c}\frac{ |u |^p}{|x|^p}.
\end{align*}
Choosing $\beta$ as in \eqref{beta} we thus arrive again at \eqref{hardy-outside}. 
\end{proof}

\begin{remark}
{\bf (i)}. Note that for $\sigma$ large enough we have $C(\sigma, R) = \left(\frac{p-n}{p}\right)^p$ which is in modulus equal to the sharp constant in the inequality \eqref{hardy-rn-class} valid in the case $n>p$. On the other hand, for $\sigma$ small enough we have $C(\sigma, R)=R^{\, p}\, \sigma^{\frac{p}{p-1}} $ which vanishes in the limit $\sigma\to 0$, as expected. 

\smallskip

\noindent {\bf (ii)}. In the case $n=p$ we have $C(\sigma, R)=0$ which is natural since the inequality 
\begin{equation}\label{p=n}
\int_{B_R^c} |\nabla u |^n  + \sigma\! \int_{\partial B_R} |u|^n  \ \geq \ C  \int_{B_R^c} \frac{|u|^n}{|x|^n} \ , \qquad u \in W^{1,n}(B_R^c) 
\end{equation}
fails for any $C>0$ independently of $R$ and $\sigma$. To see this, consider the family of test functions
$$
u_k(x) = \left(1- \frac{\log (|x|/R)}{\log k}\right)_+ \qquad k\in \N, \quad x\in B_R^c. 
$$
By inserting $u_k$ into \eqref{p=n} and letting $k\to \infty$ it follows that \eqref{p=n} must fail whenever $C>0$. This is closely related to \cite[Ex.~2]{MMP} which shows that if $\Omega= B_R^c$ and $p=n$, then the best constant in the hardy inequality \eqref{hardy-general} is zero, i.e.~$\mu_p(B_R^c) =0$, see equation \eqref{var-prob-dir}. 
\end{remark}

\medskip


\section*{Acknowledgements}
T.~E. has been supported by the Swedish Research Council grant Nr. 2009-6073 and
H.~K. has been supported by the grant MIUR-PRIN 2010-11 number 2010A2TFX2 for the project ''Calcolo delle variazioni''.


\end{document}